\documentclass[10pt,twoside]{siamltex}
\usepackage{amsfonts,epsfig}
\usepackage{amsmath}
\usepackage{amssymb}
\usepackage{enumitem}
\def\cvd{~\vbox{\hrule\hbox{%
     \vrule height1.3ex\hskip0.8ex\vrule}\hrule } }

\newtheorem{remark}[theorem]{Remark}

\def\fix{\mbox{\rm fix}}
\def\supp{\mbox{\rm supp}}
\def\Sym{\mbox{\rm Sym}}
\def\Alt{\mbox{\rm Alt}}
\def\N{\mathbb{N}}

\usepackage[utf8]{inputenc}

\title{Bochert's results on the minimal degree of multiply transitive permutation groups}
\author{Bernd Schomburg\thanks{P.O. Box 1644, 53734 Sankt Augustin, Germany (bernd.schomburg@me.com).}}
\begin{document}

\maketitle
\begin{abstract}
We give a modern account of Alfred Bochert's results from 1892 on the minimal degree of doubly, triply and quadruply transitive permutation groups. 
\end{abstract}


\section{Introduction}
Let $\Omega$ be a finite set and $G \le \Sym(\Omega)$ be a permutation group\footnote{We use the notation as in \cite{dixon-1996}. For disjoint set unions we will use the $+$ and $\sum$ symbols.} on $\Omega$. The minimal degree $m$ of $G$ is the least degree of its non-identity elements, i.e. $m= \min \{|\supp(u)| \mid u \in G, u \ne 1\}$.
In 1892, A. Bochert published lower bounds for the minimal degree of non-trivial doubly, triply and quadruply transitive permutation groups \cite{bochert-1892b}, see also \cite[pp. 52-54]{seguier-1912}.  Although his method of proof was broadly acknowledged as outstanding (see e.g. \cite[p. 42]{wielandt-1964} and \cite[p. 314]{babai-1987}), the proofs have never been properly accounted for in modern publications. The purpose of this paper is to fill this gap.

Since 1892 research in this field has taken various directions. 
In 1897, in his last publication at all \cite{bochert-1897b}, Bochert improved the bound for non-trivial doubly transitive groups from $m \ge n/4$ (where $n = |\Omega|$) to $m \ge n/3 - \kappa \sqrt{n}$ with a suitable $\kappa > 0$.  Between 1914  and 1933 \cite{manning-1914}--\cite{manning-1933}, W.A. Manning proved lower bounds for minimal degrees and minimal $p$-degrees\footnote{For a prime divisor $p$ of the group order, the minimal $p$-degree is the smallest number of points moved by a non-identity $p$-element of the group.} for higher degrees of transitivity. Nowadays, these classical results are seen in the light of the classification of finite simple groups which implies
 that there are no non-trivial $t$-transitive groups for $t \ge 6$ and only four non-trivial 4-transitive groups \cite[p. 110]{cameron-1999} and that
 the minimal degree of a non-trivial doubly transitive group is always $\ge n/3$ \cite{hoechsmann-1999}. This explains why, with regard to minimal degrees, interest has shifted away from multiply transitive to  (simply transitive) primitive and semiprimitive permutation groups, cf.  \cite{herzog-1976}, \cite{liebeck-1984},
 \cite{liebeck-1991}, \cite{guralnik-1998} and \cite{morgan-2020}.

\section{Some combinatorial lemmas} The basic strategy already deployed by C. Jordan and his contemporaries to achieve results on minimal degrees is to start with a group element of minimal degree and to use commutators to construct new non-trivial elements related to the original one in a controlled way. We recall that if $u,v \in \Sym(\Omega)$ and $\Delta = \supp(u) \cap \supp(v)$ then
\begin{equation}
 \supp([u,v]) \subseteq \Delta \cup \Delta^u \cup \Delta^v,
\end{equation}
where $[u,v] = uvu^{-1}v^{-1}$. Since $ \Delta \cup \Delta^u \cup \Delta^v$ can be written as the disjoint union
$\Delta + (\Delta^u \setminus \Delta \cap \Delta^u)
 + (\Delta^v \setminus \Delta \cap \Delta^v)
$,
we have 
\begin{equation}
 |\supp([u,v])| \leq 3 |\Delta| - |\Delta \cap \Delta^u| - |\Delta \cap \Delta^v|.   
\end{equation}
We will also need that
\begin{equation}
\supp([u,v]) \subseteq 
 \Delta\cup \{ \alpha \in \fix(u) \mid  \alpha^v \in \Delta\} \cup \{\alpha \in \fix(v) \mid\alpha^u \in \Delta\}.  
\end{equation}

Indeed, if $\alpha \in \supp([u,v])$ then $\alpha$ cannot be both in $\fix(u)$ and $\fix(v)$. Assume that $\alpha \not\in \Delta$ and that $\alpha \in \fix(u)$ so that $\alpha^v \ne \alpha$. If $\alpha^v$ were fixed by $u$, then
$$
\alpha^{[u,v]} = \alpha^{uvu^{-1}v^{-1}} = \alpha^{vu^{-1}v^{-1}} = \alpha^{vu u^{-1}v^{-1}} = \alpha,
$$
contradicting the assumption that $\alpha \in \supp([u,v])$. Hence $\alpha^v \in \Delta$. Analogously, 
$\alpha \not\in \Delta$ and $\alpha \in \fix(v)$ imply that $\alpha^u \in \Delta$.

\begin{lemma} Let $u,v \in \Sym(\Omega)$, $\Phi\subseteq \fix([u,v]) \cap \supp(u)$ and
$\Psi \subseteq \supp(v u v^{-1})\cap \supp(u)$. Then
$$
|\supp([u,v])| \le 2|\supp(u)| - |\Phi| - |\Psi|.
$$
\end{lemma}
\begin{proof} For $x,y \in \Sym(\Omega)$ we have
$
\supp(xy) \subseteq \supp(x) \cup \supp(y)
$.
If now $\Phi \subseteq \fix(xy) \cap \supp(x)$ and $\Psi \subseteq \supp(x) \cap \supp(y)$,
then 
$$
\supp(xy) \subseteq (\supp(x) \setminus \Phi) \cup (\supp(y) \setminus \supp(x))
\subseteq (\supp(x) \setminus \Phi) \cup (\supp(y) \setminus \Psi),
$$
thus $
|\supp(xy)| \subseteq |\supp(x)| + |\supp(y)| -  |\Phi| -|\Psi|
$.
By taking $x = v u v^{-1}$ and $y = u^{-1}$ the assertion follows.
\end{proof}

The proof of the following theorem illustrates the above-mentioned general strategy for
obtaining bounds for minimal degrees.

\begin{theorem} (Jordan) Let $G$ be $t$-transitive, $t \ge 2$. If its minimal degree $m$ is greater than $3$, then
$m \ge 2t-2$.
\end{theorem}
\begin{proof} (Cf. \cite[pp. 46--47]{seguier-1912}) Take $u \in G$ such that $|\supp(u)| =m$.  Then $m >t$; for otherwise let $\alpha \in \supp(u)$, $\beta \in \fix(u)$ and $v \in G$ such that $\alpha^v = \beta$ and $\gamma^v = \gamma$ for all $\gamma \in \supp(u) \setminus \{\alpha\}$ (using that $G$ is $t$-transitive). Then the $3$-cycle $(\alpha ~\beta ~\alpha^u) =[u,v]$ lies in $G$, contrary to hypothesis.
Next we can assume w.l.o.g. that $u$ has prime order $p$. Write $t-1 = N p +r $ with $N \in \N$ and $r \in \{0,...,p-1\}$. Let $\Phi$ be the union of (the supports of)  $N$ $p$-cycles of $u$ so that
$ \Phi \subset \supp(u)$, $|\Phi| =t-1 -r$ and $\Phi^u = \Phi$. Let $\alpha \in \supp(u) \setminus \Phi$.
We distinguish two cases:\\
{\bf Case 1.} $r=0$.  Choose $\beta \in \fix(u)$. Since $G$ is $t$-transitive, there is a $v \in G$ such that 
$\phi^v = \phi$ for all $\phi \in \Phi$ and $\alpha^v = \beta$. Since $\alpha^{[u,v]} \ne \alpha$, $[u,v] \ne 1$ and $m \le |\supp([u,v])|$. Furthermore,
$\Phi \subseteq \fix([u,v]) \cap \supp(u)\cap \supp(vuv^{-1})$, hence $|\supp([u,v])| \le 2m - 2t+2$ by Lemma 2.1. Along with $m \le |\supp([u,v])|$ we get $m \ge 2t-2$.\\
{\bf Case 2.} $r \ne 0$. Put $\Psi = \Phi + \{\alpha_{-k} \mid k=1,...,r\}$, $\alpha_{k} = \alpha^{u^{k}}$.  Since $G$ is $t$-transitive, there is a $v \in G$ such that 
$\phi^v = \phi$ for all $\phi \in \Psi$ and $\alpha^v = \alpha^u$. Since $\alpha_{-1}^{vu} \ne \alpha_{-1}^{uv}$, $[u,v] \ne 1$ and thus $m \le |\supp([u,v])|$.
Moreover
$$ 
   \Psi \setminus \{\alpha_{-1}\} \subseteq \fix([v,u]) \cap \supp(u), ~~\Psi + \{\alpha\} \subseteq \supp(vuv^{-1}) \cap \supp(u),
$$
thus, again by Lemma 2.1,
$m \le |\supp([v,u])| \le 2m -  |\Psi \setminus \{\alpha_{-1}\}|  - |\Psi + \{\alpha\}| = 2m - 2t +2$ and
$m \ge 2t-2$.
\end{proof}

We continue with the following observation:
 
 \begin{lemma} Let $G$ act transitively on two finite sets $\Omega$ and $\Omega'$ and
 let ${\cal E} \subseteq \Omega \times \Omega'$ be invariant with respect to the product action of $G$ on  $\Omega \times \Omega'$.
 Then there are $M,M' \in \N$ such that (i) $|\{\beta \in \Omega'\mid (\alpha, \beta) \in {\cal E}\}| = M $ for all $ \alpha \in \Omega$, (ii) $ |\{\alpha \in \Omega \mid (\alpha, \beta) \in {\cal E}\}|  =  M'$ for all $\beta \in \Omega'$ and (iii)
 $M |\Omega|  =   M' |\Omega'|$.
 \end{lemma}

{\it Proof.} Let $\alpha,\bar{\alpha} \in \Omega$. Then there is a $g \in G$ such that $\alpha^g = \bar{\alpha}$ and
$\{\beta \mid (\alpha, \beta) \in {\cal E}\} \ni \gamma \mapsto \gamma^g \in \{\beta \mid (\bar{\alpha}, \beta) \in {\cal E}\}
$ 
is well-defined and bijective, which proves (i). (ii) follows similarly.
Finally,
$$
M' |\Omega'| = \sum_{\beta \in \Omega'} |\{\alpha \mid (\alpha, \beta) \in {\cal E}\}|= |{\cal E}| =
\sum_{\alpha \in \Omega} |\{\beta  \mid (\alpha,\beta) \in {\cal E}\}|
= M |\Omega|.\cvd
$$

As a consequence we note:
 \begin{lemma} Let $G$ be a finite group, $u \in G$, $H \le G$ and $E= \{g^{-1}u g \mid g \in H\}$ so that $H$ acts transitively on $E$ by conjugation. If $H$ acts transitively on a set $\Gamma$ and if ${\cal E} \subseteq E \times \Gamma$ is invariant with respect to the product action of $H$ on $E \times \Gamma$, then there is an $M \in \N$ such that  $
|\{\gamma \in \Gamma \mid (x,\gamma) \in {\cal E}\}| = M
$
for all $x \in E$ and
$$
|\{x \in E \mid (x,\gamma) \in {\cal E} \}| = |E| \frac{M}{|\Gamma|}
$$
for all $\gamma \in \Gamma$. \cvd
\end{lemma}

The following application of Lemma 2.4 is key for our proofs in Section 3.

\begin{lemma}  Let $G \le \Sym(\Omega)$ be $t$-transitive, $t \geq 2$, $u \in G$ 
, $\Delta \subseteq  \supp(u)$ and
$
    E = \{g^{-1} u g \mid g \in G_{(\Delta)}\}
$. Then, with $m = |\supp(u)|$:\\\\
(i) If $|\Delta| \le t-1$, then for $\gamma \in \Omega \setminus \Delta$
\begin{eqnarray}
 |\{x \in E \mid \gamma^x = \gamma\}| 
 & = & 
|E|  \frac{n -m}{n-|\Delta|},    \\
|\{x \in E \mid \gamma^x \ne \gamma\}| 
& = &  |E| \frac{m-|\Delta|}{n-|\Delta|}.
\end{eqnarray}
(ii) If $|\Delta| \le t-2$, then for all 
$\gamma, \delta \in \Omega \setminus \Delta$
such that $\gamma \ne \delta$
\begin{equation}
|\{x \in E \mid \gamma \in \fix(x), \delta \in \supp(x)\}| = |E| \frac{(n-m)(m-|\Delta|)}{(n-|\Delta|)(n- |\Delta|-1)}.
\end{equation}
(iii) If $|\Delta|=1$, then for all $\gamma \in \Omega \setminus \Delta$
\begin{equation}
|\{x \in E \mid \gamma^x \in \Delta \} |  = |E| \frac{1}{n-1}.    
\end{equation}
(iv) If $|\Delta|=1$ and $t \ge 3$, then for all 
$\gamma, \delta \in \Omega \setminus \Delta$
such that $\gamma \ne \delta$
\begin{equation}
  |\{x\in E \mid \gamma^x = \delta\}| = |E|\frac{m-2}{(n-1)(n-2)} .
\end{equation}
\end{lemma}

{\it Proof.} (i) Let $\Gamma =\Omega \setminus \Delta$ and ${\cal E} = \{ (x, \gamma) \in E \times \Gamma \mid \gamma^x = \gamma\}$. Since $|\Delta| \le t-1$, $G_{(\Delta)}$ is transitive on $\Gamma$. Moreover, ${\cal E}$ is invariant w.r.t. $G_{(\Delta)}$.
Now $\Delta \subseteq \supp(x)$ for all $x \in E$ so that $\{\gamma \in \Gamma \mid \gamma^x = \gamma\} = \fix(x)$ and $|\{\gamma \in \Gamma \mid \gamma^x = \gamma\}| = n-m$
for all $x \in E$. (2.4) follows from Lemma 2.4, as in turn (2.4) implies (2.5).\\\\
The other cases are proved along the same line. We will provide appropriate sets $\Gamma$ and ${\cal E} \subseteq E \times \Gamma$ such that ${\cal E}$ is invariant w.r.t. $G_{(\Delta)}$ and, by the conditions on $t$ and $|\Delta|$, $G_{(\Delta)}$ is transitive on $\Gamma$. We will then determine $M = |\{\gamma \in \Gamma \mid (x,\gamma) \in{\cal E} \}|$ and the respective assertion will follow from Lemma 2.4.\\\\
(ii) $\Gamma = (\Omega \setminus \Delta)^{[2]}$, ${\cal E} = 
\{ (x, (\gamma,\delta)) \in E \times \Gamma \mid \gamma \in \fix(x), \delta \in \supp(x)\}
$. For all $x \in E$,
$$
M = |\fix(x) \times (\supp(x) \setminus \Delta)| = (n-m)(m-|\Delta|).
$$
\\
(iii) $\Gamma = \Omega \setminus \Delta$,
${\cal E} = \{ (x, \gamma) \in E \times \Gamma \mid \gamma^x \in \Delta\}$.
$M= |\{\gamma \in \Gamma \mid \gamma^x \in  \Delta \}| = 1$ for all $x \in E$.\\\\
(iv)  $\Gamma = (\Omega \setminus \Delta)^{[2]}$, 
$
{\cal E} = \{ (x, (\gamma,\delta)) \in E \times \Gamma \mid \gamma^x = \delta\}
$.
For all $x \in E$,
$$
M  = |\{ \gamma \in \Omega \setminus \Delta \mid  \gamma \in \supp(x), \gamma^x \not\in \Delta\}| 
 = |\supp(x) \setminus (\Delta \cup \Delta^{x^{-1}})\}| = m-2. \cvd
$$

\section{Bochert's theorems} We are now in the position to prove Bochert's estimates. Let $\Omega$ be a finite set with $n$ elements and $G \le \Sym(\Omega)$ be $t$-transitive, $t \ge 2$. We assume the $G$ does not contain $\Alt(\Omega)$, so that its minimal degree $m$ is $ >3$, see \cite[Theorem 3.3A]{dixon-1996}.    We put 
$
\Delta_x = \supp(x) \cap \supp(u)
$ 
for $x \in G$.

\begin{theorem} If $t \ge 2$, then $m \ge n/4$ for $n \ge 38$.
\end{theorem}
\begin{proof} Let $u \in G$ be an element of minimal degree, i.e. $|\supp(u)| =m$, $\alpha \in \supp(u)$ and $\beta =\alpha^u$. Put $E= \{g^{-1}ug \mid g \in G_\alpha\}$, 
$F 
= \{x \in E \mid \beta^x = \beta\}
$ and $
{\cal F} = \{ (x, \gamma) \mid x \in F, \gamma \in \Delta_x \}
$. By (2.4),
\begin{equation}
 |F| = |E| \frac{n-m}{n-1}.
\end{equation}
Moreover,  
$xu \ne ux$ for all $x  \in F$, so that, by (2.1), $
|\Delta_x| \ge m/3$ for all $x \in F$ and thus  $|{\cal F}| \ge |F| m/3$.   
Now
${\cal F}$
can be written in the form
\begin{eqnarray*}
 {\cal F} & = &  F \times \{\alpha\} + \sum_{\gamma \in \Lambda}  \{  x \in F \mid \gamma \in \supp(x) \} \times \{\gamma\},
\end{eqnarray*}
where $\Lambda =\supp(u) \setminus \{\alpha, \beta\}$. By (2.5),
\begin{eqnarray*}
 |{\cal F}|  & = & |F| + \sum_{\gamma \in \Lambda} |\{  x \in F \mid \gamma \in \supp(x) \}| \\
 & \le & |F| + \sum_{\gamma \in \Lambda} |\{  x \in E \mid \gamma \in \supp(x) \}| \\
 & = & |F| + (m-2) |E| \frac{m-1}{n-1}.
\end{eqnarray*}
Hence, by (3.1),
$$
(m-2) |E| \frac{m-1}{n-1} \ge |{\cal F}| - |F| \ge \left(\frac{m}{3} -1\right) |E| \frac{n-m}{n-1}.
$$
This can be simplified to
$$
 n \le  3 \frac{(m-1)(m-2)}{m-3} +m = 4m + \frac{6}{m-3}.
$$
For $n \ge 38$ we must have $m \ge 10$ which shows that $n \le 4m + 6/7 $, i.e. $m \ge n/4$.
\end{proof}

\begin{theorem} If $t \ge 3$, then $m \ge n/3$ for $n \ge 23$.
\end{theorem}

\begin{proof} Let $u \in G$ of degree $m$, $\alpha \in \supp(u)$, $\beta \in \fix(u)$. Since $G$ is $2$-transitive, there is an $h \in G_\alpha$  such that $\beta^h = \alpha^u$, and $v = h u h^{-1}$ maps $\alpha$ to $\beta$. Put $
E = \{g^{-1}v g \mid g \in G_{\alpha\beta}\}$.
Then  $xu \ne ux$ for all $x \in E$ and
$$
{\cal E} = \{(x,\gamma) \mid x \in E, \gamma \in \supp([u,x]) \}
$$
has at least $|E| m$ elements. 
We consider
\begin{eqnarray*}
{\cal F} & = & \{ (x, \gamma) \mid x \in E, \gamma \in \Delta_x \}, \\
{\cal G} & = & \{ (x, \gamma) \mid x \in E, \gamma \in \Delta_x  \cap \Delta_x^u \} = \{(x, \gamma) \mid x \in E, \{\gamma, \gamma^{u^{-1}}\} \subseteq  \Delta_x \},
\end{eqnarray*}
so that by (2.2)
$
|{\cal E}| \le 3 |{\cal F}| -| {\cal G}|
$. 
Now 
$$
{\cal F} =  E \times \{\alpha\} + \sum_{\gamma \in \Lambda}  \{  x \in E \mid \gamma \in \supp(x) \} \times \{\gamma\},
$$
with $\Lambda =\supp(u) \setminus \{\alpha\}$, and by (2.5),
\begin{equation}
|{\cal F}| = |E| + (m-1) \frac{|E|(m-2)}{n-2} = |E| \left(1 + \frac{(m-1)(m-2)}{n-2} \right).
\end{equation}
Next note that
\begin{eqnarray*}
{\cal G} & = &  \{(x, \gamma) \mid x \in E, \{\gamma, \gamma^{u^{-1}}\} \subseteq \Delta_x \}\\
& = & \{(x, \gamma) \mid x \in E, \gamma \in \supp(u), \{\gamma, \gamma^{u^{-1}} \} \subseteq \supp(x) \}.
\end{eqnarray*}
With $\delta_{-} = \alpha^{u^{-1}}$ and $\delta_{+} = \alpha^u$ we obtain 
$$
{\cal G} \supseteq \{(x,\alpha) \mid x \in E, \delta_{-} \in \supp(x)\}
+ \{(x,\delta_+) \mid x \in E, \delta_{+} \in \supp(x)\}.
$$
Now $\{\delta_+, \delta_-\} \cap \{\alpha,\beta\} = \emptyset$, so that by (2.2), applied to $v$,
$$
|{\cal G}| \ge |\{x \in E \mid  \delta_{-} \in \supp(x)\}|
+ |\{ x \in E \mid \delta_{+} \in \supp(x)\}| = 2 |E| \frac{m-2}{n-2}.
$$
Finally,
$$
m |E| \le |E| \left(3 + (3(m-1) -2)\frac{m-2}{n-2}\right),
$$
i.e.
$$
m \le 3 + (3m-5)\frac{m-2}{n-2},
$$
which is easily seen to be equivalent to
$$
n \le 3m + \frac{4}{m-3}.
$$
If $n \ge 23$, then $m \ge 8$ so that $n \le 3m + 4/5 $, i.e. $m \ge n/3$.
\end{proof}

The following version of Bochert's bound for $t\ge 4$ is slightly weaker than his original result. We follow \cite[pp. 7--11]{Parker-1957}.
\begin{theorem} If $t \ge 4$ then $m \ge 6$ and $n-3 \le 2m$.
\end{theorem}

\begin{proof} Theorem 2.2 shows that $m \ge 6$. Let $u \in G$ be such that  $|\supp(u)| =m$, $\alpha \in \supp(u)$, $\beta =\alpha^u$ and $\Lambda =\supp(u) \setminus \{\alpha, \beta\}$. Since $G$ is doubly transitive there is an $h \in G$ such that 
$\alpha^h \in \fix(u)$ and $\beta^h \in \Lambda$.
Then 
$v = h u h^{-1} \in G_\alpha \setminus G_\beta$, $|\supp(v)|=m$ 
and $E = \{ g^{-1}v g \mid g \in G_{\alpha \beta}\}\subseteq G_\alpha \setminus G_\beta$. Hence $[u,x] \ne 1$ for all $x \in E $ and
$
{\cal E} =  \{ (x, \gamma) \mid x \in E, \gamma \in \supp([u,x]) \}
$
has at least $|E| m$ elements.  By (2.3),
$
{\cal E} 
\subseteq  {\cal F} \cup {\cal G} \cup {\cal H}
$
with
\begin{eqnarray*}
{\cal F} & = & \{ (x, \gamma) \mid x\in E, \gamma \in \Delta_x\},\\
{\cal G} & = & \{ (x, \gamma) \mid x\in E, \gamma^u = \gamma, \gamma^x \in \Delta_x\},\\
{\cal H} & = & \{ (x, \gamma) \mid x\in E, \gamma^x = \gamma, \gamma^u \in \Delta_x \}.
\end{eqnarray*}
 As in proof of Theorem 3.2 it is shown that $|{\cal F}|$ is given by (3.2). In order to get an upper bound for $|{\cal G}|$ we note that 
\begin{eqnarray*}
{\cal G} & \subseteq &\{ (x, \gamma) \mid x \in E, \gamma \in \fix(u), \gamma^x \in \supp(u) \setminus \{\alpha\} \}\\
&=& \sum_{\gamma \in \fix(u)} (\{ (x, \gamma) \mid x \in E,  \gamma^x \in \Lambda\} +  \{ (x, \gamma) \mid x \in E, \gamma^x = \beta \}).
\end{eqnarray*}
Now, applying (2.8) to $\Omega \setminus \{\alpha\}$ and $\Delta = \{\beta\}$, we get
\begin{eqnarray*}
|\{ (x, \gamma) \mid x \in E, \gamma^x \in \Lambda\}| & \le &
 \sum_{\delta \in \Lambda} |\{x \in E \mid \gamma^x = \delta\}| \\& = &
 |\Lambda| \frac{|E|(m-2)}{(n-2)(n-3)},
\end{eqnarray*} 
and, by (2.7),
$$
|\{ (x, \gamma) \mid x \in E, \gamma^x = \beta\}| = \frac{|E|}{n-2}.
$$
Since $|\fix(u)| = n-m$,
$$
|{\cal G}| \le |E| \frac{n-m}{n-2} \left( \frac{(m-2)^2}{n-3}+ 1\right).
$$
Finally,
\begin{eqnarray*}
{\cal H} & = & \{ (x, \gamma) \mid x\in E, \gamma^x = \gamma, \gamma \in \supp(u),
\gamma^u \in \supp(x)\}\\
& = & E \times \{\alpha\} +\ \sum_{\gamma \in \Lambda}  \{ (x, \gamma) \mid x\in E, \gamma \in \fix(x), \
\gamma^u \in \supp(x)\}.
\end{eqnarray*}
By assumption, $G$ is $4$-transitive, hence, using (2.6), we have
\begin{eqnarray*}
|{\cal H}|  & \le &
 |E| + \sum_{\gamma \in \Lambda} |\{x \in E \mid  \gamma \in \fix(x), \gamma^u \in \supp(x)\}| \\
 & \le & |E| \left(1 + |\Lambda|  \frac{(n-m)(m-2)}{(n-2)(n- 3)}\right)\\
  & = & |E| \left(1 +  \frac{(n-m)(m-2)^2}{(n-2)(n- 3)}\right).
\end{eqnarray*} 
Therefore
\begin{eqnarray*}
m  & \le &  \left(1 + \frac{(m-1)(m-2)}{n-2}  \right) +  \left( \frac{(n-m)(m-2)^2}{(n-2)(n-3)}+ \frac{n-m}{n-2}\right) \\
    & & + \left(1 +  \frac{(n-m)(m-2)^2}{(n-2)(n- 3)} \right)
\end{eqnarray*} 
which is equvalent to
$$
m-3    \le  \frac{(m-2)^2}{n-2} + 2\frac{(n-m)(m-2)^2}{(n-2)(n-3)},
$$
i.e., if we put $M = m-3$ and $N = n-3$, 
$$
M   \le  \frac{(M+1)^2}{N+1} + 2\frac{(N-M)(M+1)^2}{N(N+1)},
$$
or, after rearrangement,
$$
N(N+1)  \le 3 \frac{(M+1)^2}{M} N- 2(M+1)^2.
$$
Putting $N_0 = (3(M+1)^2 -M)/2M$ and completing the square, we obtain
$$
 \left(N - N_0\right)^2  \le   N_0^2 -  2 (M+1)^2
  =  \frac{p(M)}{4M^2},
$$
so that
\begin{equation}
N \le N_0 + \frac{\sqrt{p(M)}}{2M},
\end{equation}
where $p(x) = x^4+ 14x^3 + 35 x^2 + 30x +9$. Since $M = m-3 \ge 3$,   $p(M) < (M^2 +7M)^2$ and, by (3.3),
$$
N < N_0 + \frac{M^2+ 7M}{2M} = 2M +6+ \frac{3}{2 M}.
$$
This implies $n - 3 \le 2m$.
\end{proof}

\begin{remark} {\rm (i) Bochert's original result states that $m \ge (n/2) -1$ for $t \ge 4$. The above proof can be amended to show that this bound in fact holds for all $n \ge 27$. (ii) By the classification of finite simple groups, the Mathieu groups $M_{11}$, $M_{12}$, $M_{23}$ and $M_{24}$ are the only non-trivial $4$-transitive groups ($M_{12}$ and $M_{24}$ being actually $5$-transitive) \cite[p. 110]{cameron-1999}. The minimal degree of $M_{11}$ and $M_{12}$ is $8$, that of $M_{23}$ and $M_{24}$ $16$. The respective lower bounds provided by Theorem 3.3 are $6$, $6$, $10$ and $11$.}

\end{remark}

\section*{Acknowledgment} The author thanks Cheryl E. Praeger for her helpful comments.

\bibliographystyle{plain}

\begin{thebibliography}{99}
\bibitem{babai-1987} L\'aszl\'o Babai and \'Akos Seress. {\em On the Degree of Transitivity of Permutation Groups:
A Short Proof.}
J. Combinatorial Theory, Series A 45, 310-315 ( 1987)
\bibitem{bochert-1892b} Alfred Bochert. {\em Ueber die Classe der transitiven Substitutionengruppen.} Math. Ann.  40  (1892),  176--193.
\bibitem{bochert-1897b} Alfred Bochert. {\em Ueber die Classe der transitiven Substitutionengruppen. II.}
Math. Ann.  49  (1897),  133--144.
\bibitem{cameron-1999} Peter J. Cameron.  
\newblock  {\em Permutation Groups.} 
\newblock 
Cambridge University Press, Cambridge etc.,  1999.
\bibitem{dixon-1996} John D. Dixon and Brian Mortimer.  
\newblock  {\em Permutation Groups.} 
\newblock 
Springer, New York etc.,  1996.
\bibitem{guralnik-1998} Robert M. Guralnick and Kay Magaard. {\em 
On the minimal degree of a primitive permutation group.}
J. Algebra 207 (1998),  127–145.
\bibitem{herzog-1976} Marcel Herzog and Cheryl E. Praeger. 
{\em Minimal degree of primitive permutation groups.}
Combinatorial mathematics, IV (Proc. Fourth Australian Conf., Univ. Adelaide,
Adelaide, 1975), 116–122. Lecture Notes in Math., Vol. 560, Springer, Berlin, 1976.
\bibitem{hoechsmann-1999} Jens 
H\"ochsmann. {\em On minimal p-degrees in 2-transitive permutation groups.} Arch. Math. 72 (1999), 405–417. 
\bibitem{liebeck-1984} Martin W. Liebeck. {\em 
On minimal degrees and base sizes of primitive permutation groups.}
Arch. Math. 43 (1984), 11–15.
\bibitem{liebeck-1991} Martin W. Liebeck and Jan Saxl.
{\em Minimal degrees of primitive permutation groups, with an application to
monodromy groups of covers of Riemann surfaces.}
Proc. London Math. Soc. (3) 63 (1991),  266–314.
\bibitem{manning-1914}  William A. Manning. {\em On the class of doubly transitive groups.} Bull. Amer. Math. Soc.  20  (1914),  468--476.
\bibitem{manning-1917}  William A. Manning. {\em The degree and class of multiply transitive groups.} Trans. Amer. Math. Soc. 18 (1917),  463--479.
\bibitem{manning-1929} William A. Manning. {\em The degree and class of multiply transitive groups. II.} Trans. Amer. Math. Soc. 31 (1929),  643--653.
\bibitem{manning-1933} William A. Manning. {\em The degree and class of multiply transitive 
groups.III.} Trans. Amer. Math. Soc. 35 (1933), 585--599. 
\bibitem{morgan-2020} Luke Morgan,  Cheryl E. Praeger and Kyle Rosa. {\em Bounds for finite semiprimitive permutation groups: order, base size, and minimal degree.} 
Proc. Edinb. Math. Soc. (2) 63 (2020), 1071–1091.
\bibitem{Parker-1957} Ernest T. Parker. {\em On Quadruply Transitive Groups.} Dissertation, Ohio State University, 1957.
\bibitem{seguier-1912} Jean Armand Marie Joseph de S\'eguier. {\em Th\'eorie des groupes finis: \'Elements de la th\'eorie des groupes de substitutions.} Gauthier-Villars, Paris, 1912.
\bibitem{wielandt-1964} Helmut Wielandt.  \newblock  {\em Finite Permutation Groups.} 
\newblock 
Academic Press, New York and London,  1964.

\end{thebibliography}

\end{document}